\documentclass[12pt,english,oneside,reqno]{amsart}
\usepackage[T2A]{fontenc}
\usepackage[utf8]{inputenc}
\usepackage{amstext}
\usepackage{amssymb}
\usepackage{amsmath}
\usepackage[english]{babel}
\usepackage{amsthm}
\newcounter{defen}	 
 \newtheorem{defn}[defen]{Definition}
\theoremstyle{plain}
\newcounter{lem}
  \newtheorem{lemma}[lem]{Lemma}
\newcounter{thm}
\newtheorem{theorem}[thm]{Theorem}
\newcounter{cor}
  \newtheorem{corollary}[cor]{Corollary}
  \theoremstyle{remark}
\newcounter{claimo}
  \newtheorem{claim}[claimo]{Claim}
\begin{document}

\title{Complete metric on mixing actions of general groups\thanks{The work was completed in the framework of the Program for Support
of Leading Scientific Schools of the Russian Federation (grant no.
NS-3038.2008.1) }
}

\author{Tikhonov S.V.}

\email{tikhonovc@mail.ru}



\begin{abstract}
In this paper the metric on the set of mixing actions of a countable
infinite group is introduced so that the corresponding space is complete
and separable.
\end{abstract}
\keywords{monotilable group\and measure preserving transformations\and mixing group actions}
\maketitle
\section{Introduction}
\label{intro}
Let $\mathcal{A}$ be the group of invertible measure preserving transformations
of Lebesgue space $\left(X,\Sigma,\mu\right)$, endowed with weak
topology(see below). Continuous homomorphism of a topological group
$\mathcal{G}$ into $\mathcal{A}$ is called $\mathcal{G}$-action(or
simply, action). 

Measure preserving action $T$ of a locally compact group is mixing if for arbitrary measurable
sets $A$ and \textbf{$B\subset X$} 
\[
\mu\left(T^{g}A\cap B\right)\rightarrow\mu\left(A\right)\mu\left(B\right)\,\text{as}\,\, g\rightarrow\infty;
\]
if $\mathcal{G}$ is finitely generated then one can say about length
of its elements, denoted by $\left|\bullet\right|$, and $g\rightarrow\infty$
iff $\left|g\right|\rightarrow\infty$. 

The aim of the paper is to elaborate certain machinery for study of mixing
$\mathcal{G}$-actions from the point of view going back to the well-known
theorem due to Halmos and Rokhlin about conjugacy class of aperiodic transformation. 

This question for $\mathbb{Z}$-actions was considered in \cite{TikhonovMSB2}. 

Here the general case of countable infinite groups and their actions
is considered:

--- the metric on the set of mixing $\mathcal{G}$-actions is introduced
so that the corresponding space $\mathcal{M}_{\mathcal{G}}$ is complete
and separable,

--- it is shown that for generic action its conjugates are everywhere
dense in this space. 

Moreover, if $\mathcal{G}$ is finitely generated infinite monotilable
amenable group then conjugates to Cartesian product of two free mixing
actions \emph{are everywhere dense in $\mathcal{M}_{\mathcal{G}}$.}

\section{The metric}
\label{sec:1}

Let $\left(X,\Sigma,\mu\right)$ be separable Lebesgue space and $\left\{ A_{i}\right\} $
be a countable collection of sets generating the $\sigma$-algebra
$\Sigma$. Measure $\mu$ is supposed to be normalized and continuous.
All the spaces with these properties are isomorphic to each other;
in particular such a space is isomorphic to its Cartesian square.
The set $\mathcal{A}$ of invertible measure preserving transformations
$X$ is separable with respect to \emph{weak topology} defined by
the base of neighbourhoods 
\[
\mathcal{U}\left(T,q,\varepsilon\right)=\left\{ S\mid\forall A,B\in q:\left|\mu\left(TA\cap B\right)-\mu\left(SA\cap B\right)\right|<\varepsilon\right\} ,
\]
(here and below $q$ stand for finite subset of $\left\{ A_{i}\right\} $).
Moreover $\mathcal{A}$ is a topological group (i.e. multiplication
and inversion are continuous). The weak topology is generated by any
of the metrics 
\[
d\left(T,S\right)=\sum_{i\in\mathbb{N}}\frac{1}{2^{i}}\left(\mu\left(TA_{i}\bigtriangleup SA_{i}\right)+\mu\left(T^{-1}A_{i}\bigtriangleup S^{-1}A_{i}\right)\right),
\]

and 
\[
a\left(T,S\right)=\sum_{i,j\in\mathbb{N}}\frac{1}{2^{i+j}}\left|\mu\left(TA_{i}\cap A_{j}\right)-\mu\left(SA_{i}\cap A_{j}\right)\right|.
\]

Metric $d$ is more preferable in the sense that the space $\mathcal{A}$
is complete%
\footnote{To speak about completeness it is necessary to identify transformations
(or $\mathcal{G}$-actions that will be the main object below), coinciding
mod 0. %
} with respect to $d$; it is called metric of weak topology. This
metric can be naturally generalized to the case of group actions.
Let $\mathcal{G}$ be countable  group; the set $\mathcal{A}_{\mathcal{G}}$
of $\mathcal{G}$-actions by measure preserving transformations is
a complete separable metric space with respect to metric 
\[
d\left(T,S\right)=\sum_{g\in\mathcal{G}}\frac{1}{2^{\left|g\right|}}d\left(T^{g},S^{g}\right),
\]
where $\left\{ \left|g\right|\right\} _{g\in\mathcal{G}}$ is collection
of positive numbers satisfying the condition $\sum_{g\in\mathcal{G}}\frac{1}{2^{\left|g\right|}}<\infty$. 

Action $T$ of a group $\mathcal{G}$ is called \emph{mixing}, if for
each pair $A,B\in\Sigma$, one has 
\[
\mu\left(T^{g}A\cap B\right)\rightarrow\mu\left(A\right)\mu\left(B\right)
\]
 as $\left|g\right|\rightarrow\infty$. 

The set $\mathcal{M_{G}}$ of mixing actions of $\mathcal{G}$ is
endowed by \emph{lead metric} 
\begin{equation}
m\left(T,S\right)=d\left(T,S\right)+\sup_{g\in\mathcal{G}}a\left(T^{g},S^{g}\right).\label{eq:TheMetrix}
\end{equation}

If $\mathcal{G}=\mathbb{Z}$ the metric $m$ coincides with complete
separable metric introduced in \cite{TikhonovMSB2}. Metric $m\left(T,S\right)$
can be considered on the set $\mathcal{A}_{\mathcal{G}}$, however
it is not separable (even in the case $\mathcal{G}=\mathbb{Z}$).
\begin{lemma}
\label{lem:Topology}Metrics $m$ and $w,w=\sup_{g\in\mathcal{G}}a\left(T^{g},S^{g}\right)$
generate in $\mathcal{M_{G}}$ the same topology with the base of
neighbourhoods 
\[
\mathcal{Q}\left(T,q,\varepsilon\right)=\left\{ S\mid\forall A,B\in q,\sup_{g\in\mathcal{G}}\left|\mu\left(T^{g}A\cap B\right)-\mu\left(S^{g}A\cap B\right)\right|<\varepsilon\right\} ,
\]
 where $T\in\mathcal{M_{G}}$, $q\subset\left\{ A_{i}\right\} $ is
finite and $\varepsilon\in\mathbb{R}^{+}\setminus\left\{ 0\right\} $.\end{lemma}
\begin{proof}
Coincidence of topologies follows from the fact that metrics $d$
and $a$ generate the same topology (see \cite{TikhonovMSB2}). Now we
show that sets $\mathcal{Q}\left(T,q,\varepsilon\right)$ form a base
of topology generated by $w$. 

1. In every $\left(w,\delta\right)$-neighbourhood of the action $T$
there is a neighbourhood $\mathcal{Q}\left(T,q,\frac{\delta}{4}\right)$,
where $q$ is such that $2\sum_{A_{i}\notin q}\frac{1}{2^{i}}<\frac{\delta}{2}$.
In fact, one has 
\[
S\in\mathcal{Q}\left(T,q,\frac{\delta}{4}\right)\Rightarrow\forall g:a\left(T^{g},S^{g}\right)=\sum_{i,j\in\mathbb{N}}\frac{1}{2^{i+j}}\left|\mu\left(T^{g}A_{i}\cap A_{j}\right)-\mu\left(S^{g}A_{i}\cap A_{j}\right)\right|\leqslant
\]
\[
\leqslant\sum_{A_{i},A_{j}\in q}\frac{1}{2^{i+j}}\left|\mu\left(T^{g}A_{i}\cap A_{j}\right)-\mu\left(S^{g}A_{i}\cap A_{j}\right)\right|+\sum_{A_{i}\notin q}\frac{1}{2^{i}}+\sum_{A_{j}\notin q}\frac{1}{2^{j}}\leqslant
\]
\[
\leqslant\sum_{A_{i},A_{j}\in q}\frac{1}{2^{i+j}}\left|\mu\left(T^{g}A_{i}\cap A_{j}\right)-\mu\left(S^{g}A_{i}\cap A_{j}\right)\right|+\frac{\delta}{2}\leqslant\frac{3\delta}{4}.
\]

Since this is true for any $g$, 
\[
w\left(T,S\right)=\sup_{g\in\mathcal{G}}a\left(T^{g},S^{g}\right)<\delta.
\]

2. In every set $\mathcal{Q}\left(T,q,\varepsilon\right)$ there is
$\left(w,\frac{\varepsilon}{2^{2l}}\right)$-neighbourhood of $T$,
where $l$ is the maximal number of elements from $q$. Then for $A_{i},A_{j}\in q$,
one has

\[
\frac{1}{2^{2l}}\left|\mu\left(S^{g}A_{i}\cap A_{j}\right)-\mu\left(T^{g}A_{i}\cap A_{j}\right)\right|\leqslant\frac{1}{2^{i+j}}\left|\mu\left(S^{g}A_{i}\cap A_{j}\right)-\mu\left(T^{g}A_{i}\cap A_{j}\right)\right|\leqslant
\]
\[
\leqslant\sum_{i,j\in\mathbb{N}}\frac{1}{2^{i+j}}\left|\mu\left(S^{g}A_{i}\cap A_{j}\right)-\mu\left(T^{g}A_{i}\cap A_{j}\right)\right|=a\left(S^{g},T^{g}\right).
\]
Hence for each $S$ such that $w\left(T,S\right)<\frac{\varepsilon}{2^{2l}}$,
every pair $A,B\in q$ and every $g\in\mathcal{G}$ the inequality
\[
\sup_{g\in\mathcal{G}}\left|\mu\left(T^{g}A\cap B\right)-\mu\left(S^{g}A\cap B\right)\right|\leqslant2^{2l}a\left(S^{g},T^{g}\right)<\varepsilon,
\]
holds and therefore $S\in\mathcal{Q}\left(T,q,\varepsilon\right)$.\end{proof}
\begin{lemma}
The space $\mathcal{M_{G}}$ is separable.\end{lemma}
\begin{proof}
It is sufficient to show that for each finite $q\subset\left\{ A_{i}\right\} $
and arbitrary $\varepsilon>0$ one can choose the countable sets of
actions $\left\{ T_{l}\right\} $ such that the space $\mathcal{M_{G}}$
has a covering by the sets $\left\{ \mathcal{Q}\left(T_{l},q,\varepsilon\right)\right\} _{l}$.

For every mixing $\mathcal{G}$-action $T$ there is a finite collection
of parameters defined as follows:

$n$ is a positive integer such that 
\[
A,B\in q,\left|g\right|>n\Rightarrow\left|\mu\left(T^{g}A\cap B\right)-\mu\left(A\right)\mu\left(B\right)\right|<\frac{\varepsilon}{2};
\]
$\left\{ n_{g}\right\} _{\left|g\right|\leqslant n}$ is the collection
of positive integers such that 
\[
A,B\in q,\left|g\right|\leqslant n\Rightarrow\left|\mu\left(T^{g}A\cap B\right)-\frac{\varepsilon n_{g}}{2}\right|<\frac{\varepsilon}{2}.
\]
The set of these parameters is countable. If actions $T$ and $S$
have the same parameter collections then for each pair $A,B\in q$,
one has 
\[
\sup_{g}\left|\mu\left(T^{g}A\cap B\right)-\mu\left(S^{g}A\cap B\right)\right|=
\]
\[
=\max\left\{ \sup_{\left|g\right|\leqslant n}\left|\mu\left(T^{g}A\cap B\right)-\mu\left(S^{g}A\cap B\right)\right|,\sup_{\left|g\right|>n}\left|\mu\left(T^{g}A\cap B\right)-\mu\left(S^{g}A\cap B\right)\right|\right\} \leqslant
\]
\[
\leqslant\max\Bigl\{\sup_{\left|g\right|\leqslant n}\left(\left|\mu\left(T^{g}A\cap B\right)-\frac{\varepsilon n_{g}}{2}\right|+\left|\mu\left(S^{g}A\cap B\right)-\frac{\varepsilon n_{g}}{2}\right|\right),
\]
\[
\sup_{\left|g\right|>n}\left(\left|\mu\left(T^{g}A\cap B\right)-\mu\left(A\right)\mu\left(B\right)\right|+\left|\mu\left(A\right)\mu\left(B\right)-\mu\left(S^{g}A\cap B\right)\right|\right)\Bigr\}<\varepsilon.
\]
Hence $S\in\mathcal{Q}\left(T,q,\varepsilon\right)$ and the required
countable set of actions can be formed by choosing one action for
each of the constructed parameter collections.\end{proof}
\begin{theorem}
$\mathcal{M_{G}}$ is a complete separable metric space.\end{theorem}
\begin{proof}
By virtue of the previous lemma it remains to check completeness.
Let $\left\{ T_{i}\right\} \subset\mathcal{M}_{\mathcal{G}}$ be a
Cauchy sequence with respect to $m$. Then it is also Cauchy in the
space $\mathcal{A_{G}}$ and converges in $\mathcal{A_{G}}$ to some
$\mathcal{G}$-action $T$. In addition $T_{i}^{g}\rightarrow T^{g}$
for every $g\in\mathcal{G}$. Since $d\left(T,T_{i}\right)\rightarrow0$
it is sufficient to check that $T$ is mixing and 
\[
\sup_{g}a\left(T^{g},T_{i}^{g}\right)\rightarrow0.
\]
Let $i\in\mathbb{N}$ be such that for any $l,m>i$ and every $g\in\mathcal{G}$
one has $a\left(T_{l}^{g},T_{m}^{g}\right)<\varepsilon.$ Now $T_{l}^{g}$
converges weakly to $T^{g}$, so there exists $n\left(g\right)>i$
such that $a\left(T_{n\left(g\right)}^{g},T^{g}\right)<\varepsilon$. 

It follows that for $m>i$, 
\[
\sup_{g}a\left(T^{g},T_{m}^{g}\right)\leqslant\sup_{g}\left(a\left(T^{g},T_{n\left(g\right)}^{g}\right)+a\left(T_{m}^{g},T_{n\left(g\right)}^{g}\right)\right)<2\varepsilon.
\]

Next, for $m$ chosen above and given sets $A_{i},A_{j}$ there is a
number $k_{0}$ such that 
\[
\left|\mu\left(T_{m}^{g}A_{i}\cap A_{j}\right)-\mu\left(A_{i}\right)\mu\left(A_{j}\right)\right|<\varepsilon,
\]
 if $\left|g\right|>k_{0}$.

Now, using the inequality 
\[
\frac{1}{2^{i+j}}\left|\mu\left(T_{m}^{g}A_{i}\cap A_{j}\right)-\mu\left(T^{g}A_{i}\cap A_{j}\right)\right|\leqslant
\]
\[\leqslant\sum_{i,j\in\mathbb{N}}\frac{1}{2^{i+j}}\left|\mu\left(T_{m}^{g}A_{i}\cap A_{j}\right)-\mu\left(T^{g}A_{i}\cap A_{j}\right)\right|=a\left(T_{m}^{g},T^{g}\right),
\]
one has 
\[
\left|\mu\left(T^{g}A_{i}\cap A_{j}\right)-\mu\left(A_{i}\right)\mu\left(A_{j}\right)\right|\leqslant
\]
\[
\leqslant\left|\mu\left(T_{m}^{g}A_{i}\cap A_{j}\right)-\mu\left(T^{g}A_{i}\cap A_{j}\right)\right|+\left|\mu\left(A_{i}\right)\mu\left(A_{j}\right)-\mu\left(T_{m}^{g}A_{i}\cap A_{j}\right)\right|<
\]
\[
<\varepsilon\left(2^{i+j+1}+1\right),
\]
whence in view of arbitrariness of $\varepsilon,i$ and $j$ it follows
that $T$ is mixing. 
\end{proof}
\section{Dense subsets}
\label{sec:2}

We call the family of actions 
\[
\left\{ U^{-1}TU\right\} _{U\in\mathcal{A}}
\]
\emph{conjugacy class }of $\mathcal{G}$-action $T$. The closure
of this class in $\mathcal{M_{G}}$ is denoted by $\overline{T}$. 
\begin{lemma}
For each $U\in\mathcal{A}$ the mapping $\widetilde{U}:T\mapsto U^{-1}TU,\mathcal{M_{\mathcal{G}}}\rightarrow\mathcal{M_{\mathcal{G}}}$
is continuous.\end{lemma}
\begin{proof}
Consider neighbourhoods $\mathcal{Q}\left(U^{-1}TU,q,\varepsilon\right)$
and $\mathcal{Q}\left(T,Uq,\varepsilon\right)$ of the actions $U^{-1}TU$
and $T$ respectively.

Then
\[
P\in\mathcal{Q}\left(T,Uq,\varepsilon\right)\Rightarrow
\]
 
\[
\Rightarrow\forall A,B\in q:\varepsilon>{\displaystyle \sup_{g\in\mathcal{G}}}\left|\mu\left(T^{g}UA\cap UB\right)-\mu\left(P^{g}UA\cap UB\right)\right|=
\]
\[
={\displaystyle \sup_{g\in\mathcal{G}}}\left|\mu\left(U^{-1}T^{g}UA\cap B\right)-\mu\left(U^{-1}P^{g}UA\cap B\right)\right|\Rightarrow
\]
\[
\Rightarrow U^{-1}PU\in\mathcal{Q}\left(U^{-1}TU,q,\varepsilon\right).
\]

\end{proof}
The $\mathcal{G}$-action $S$ is a \emph{factor} of $\mathcal{G}$-action
$T$, if there exists measure preserving mapping $v$ of the space
$\left(X,\Sigma,\mu\right)$ into itself such that  $v\circ T^{g}=S^{g}\circ v$
mod0 for every $g\in\mathcal{G}$.

\begin{lemma}
\label{lem:factorDence}If $S$ is a factor of $T$ then $\overline{S}\subset\overline{T}$.\end{lemma}
\begin{proof}
It suffices to check that $S\in\overline{T}$. Let $v$ be so that
$v\circ T^{g}=S^{g}\circ v$ mod0. Consider a neighbourhood $\mathcal{Q}\left(S,q,\varepsilon\right)$
of the action $S$ and show that in this neighbourhood there is an
action of the form $U^{-1}TU$ for certain $U\in\mathcal{A}$. We
take as $U$ any invertible measure preserving transformation of the
space $\left(X,\Sigma,\mu\right)$ transferring $A$ to $v^{-1}A$
for each $A\in q$.

Then for all $g\in\mathcal{G}$, $A,B\in q$, we have 
\[
\mu\left(S^{g}A\cap B\right)-\mu\left(U^{-1}T^{g}UA\cap B\right)=\mu\left(S^{g}A\cap B\right)-\mu\left(T^{g}UA\cap UB\right)=
\]

\[
=\mu\left(S^{g}vv^{-1}A\cap B\right)-\mu\left(T^{g}v^{-1}A\cap v^{-1}B\right)=
\]
\[
=\mu\left(v^{-1}S^{g}vv^{-1}A\cap v^{-1}B\right)-\mu\left(T^{g}v^{-1}A\cap v^{-1}B\right)=
\]
\[
=\mu\left(v^{-1}vT^{g}v^{-1}A\cap v^{-1}B\right)-\mu\left(T^{g}v^{-1}A\cap v^{-1}B\right)=0.
\]

The last equality follows from the inclusion $v^{-1}vT^{g}v^{-1}A\supset T^{g}v^{-1}A$
and the fact that measures of these sets coincide, hence they are equal
mod 0 (it is used here that $\mu\left(v^{-1}vT^{g}v^{-1}A\right)=\mu\left(v^{-1}S^{g}A\right)=\mu\left(A\right)$).\end{proof}
\begin{corollary}
\label{cor1}Actions $T$ and $S$ can be approximated arbitrarily
well by actions conjugate to $T\times S$.
\end{corollary}
Subset of metric space has \emph{type} $G_{\delta}$ if it is countable
intersection of open sets. Everywhere dense subset in a complete metric
space is called generic. The expression ''typical mixing possesses
the property $P$'' means that this $P$ hold for generic set of
mixing actions. 
\begin{theorem}
\label{thm:soprTipical}The set 
\[
W=\left\{ T\in\mathcal{M_{G}}\mid\overline{T}=\mathcal{M_{G}}\right\} 
\]
is generic. \end{theorem}
\begin{proof}
Let $\left\{ T_{i}\right\} $ be a countable dense subset in $\mathcal{M_{G}}$.
The set $W$ consists of those mixing actions that can be transferred
by conjugation into each of the sets 
\[
\mathcal{U}_{n}\left(T_{i}\right)=\left\{ S\mid m\left(T_{i},S\right)<\frac{1}{n}\right\} ,\,\,\,\, i,n\in\mathbb{N}.
\]

Thus 
\[
W=\bigcap_{i,n}\left\{ S\mid\exists U:U^{-1}SU\in\mathcal{U}_{n}\left(T_{i}\right)\right\} .
\]
 Note that all the sets right to the intersection symbol are open.
Indeed, $\widetilde{U}\left(T\right):T\mapsto U^{-1}TU$ is continuous
mapping, hence $\widetilde{U}$ --- inverse image of a neighbourhood
of the action $U^{-1}SU$ contains some neighbourhood of $S$. 

It remains to note that conjugacy class for Cartesian product of all
$T_{i}$`s is dense in $\mathcal{M_{G}}$, since it approaches arbitrarily
close to each of $T_{i}$`s(cor \ref{cor1}).
\end{proof}
\section{Everywhere dense sets of actions for monotilable amenable groups}
\label{sec:3}

In the section we consider finitely generated infinite monotilable
amenable groups. 
\begin{defn}
Group $\mathcal{G}$ is called\emph{ monotilable amenable}, if there
exists a sequence $\left\{ F_{i}\right\} _{i\in\mathbb{N}}$ of finite
subsets in $\mathcal{G}$, satisfying the following conditions: 

1. For every $g\in\mathcal{G}$, one has $\frac{\#\left(gF_{i}\cap F_{i}\right)}{\#F_{i}}\rightarrow1$
as $i\rightarrow\infty$ (here and then symbol ''$\#$'' denotes
the number of element in the set); 
\end{defn}
2. For each $i\in\mathbb{N}$, the set $F_{i}$ is a \emph{tile} in
the sense that there exists a collection $\left\{ c_{j}^{(i)}\right\} \subset\mathcal{G}$
such that $\mathcal{G}={\displaystyle \bigsqcup_{j}}F_{i}c_{j}^{(i)}$. 

We recall that an action $T$ is\emph{ free} if the set 
\[
\left\{ x\in X\mid\exists g\neq e,T^{g}x=x\right\} 
\]
 has measure zero. 

The following proposition \cite{OrnsteinWeiss87} is a group analog
of Rokhlin-Halmos lemma.
\begin{lemma}
\label{lem:Alpern}For arbitrary $\varepsilon>0$, free $\mathcal{G}$-action
$Q$ and a tile $G\subset\mathcal{G}$, there exists a set $E$ such
that $Q^{g}E,g\in G$ are disjoint and their total measure in not
less $1-\varepsilon$. 
\end{lemma}
A collection of sets $\left\{ Q^{g}E\right\} _{g\in G}$ with $E$
as in lemma \ref{lem:Alpern} is called Rokhlin tower. To mark that
its remainder $O=X\setminus\left(\cup_{g\in G}Q^{g}E\right)$ has
measure at most $\varepsilon$ we write $\left\{ Q^{g}E\right\} _{g\in G}^{\varepsilon}$.

It follows from lemma \ref{lem:Alpern} that conjugacy class of any
free $\mathcal{G}$-action is everywhere dense in $\mathcal{A}_{\mathcal{G}}$.
\begin{lemma}
\label{lemm:alpern2}For any set $C\subset\mathcal{G}$, free $\mathcal{G}$-action
$Q:Y\rightarrow Y$, $\varepsilon>0$ and a tile $F$ there exist
sets $G\subset\mathcal{G}$ and $E\in\Sigma$ such that 

1. $G=\sqcup_{i}Fc_{i},$ for some finite collection $\left\{ c_{i}\right\} \subset\mathcal{G}$,

2. $\#\left(gG\bigtriangleup G\right)<\varepsilon\#G$ for $g\in C$,

3. $\frac{3\left(\#F\right)^{2}}{\#G}<\varepsilon$,

4. $\left\{ Q^{g}E\right\} _{g\in G}^{\varepsilon}$ is a tower.\end{lemma}
\begin{proof}
From lemma \ref{lem:Alpern} and the properties
of monotilable amenable groups it follows that there exist tile $\widetilde{G}$
and set $E$ such that $\frac{3\left(\#F\right)^{2}}{\#\widetilde{G}}<\frac{\varepsilon}{2}$,
$\#\left(g\widetilde{G}\bigtriangleup\widetilde{G}\right)<\frac{\varepsilon}{8\left(\#F\right)^{2}}\#\widetilde{G}$
for $g\in C\cup FF^{-1}$ and $\left\{ Q^{g}E\right\} _{g\in\widetilde{G}}^{\frac{\varepsilon}{2}}$
is a tower.

Consider now an infinite sequence $\left\{ b_{i}\right\} $ with the
property $\mathcal{G}=\sqcup_{i}Fb_{i}$ and take as $\left\{ c_{i}\right\} $
the set $\left\{ b_{i}\mid Fb_{i}\subset\widetilde{G}\right\} $.
The set $G=\sqcup_{i}Fc_{i}$ possesses properties 1-4 above. 

Property 1 holds by construction of $G$. 

Further, 
\[
\#\left(\widetilde{G}\setminus G\right)\leqslant\#\left\{ g\in\widetilde{G}\mid\exists f,h\in F:f^{-1}g\in\left\{ b_{i}\right\} ,hf^{-1}g\notin\widetilde{G}\right\} \leqslant
\]
\[
\leqslant\sum_{f\in FF^{-1}}\#\left\{ g\in\widetilde{G}\mid fg\notin\widetilde{G}\right\} <
\]
\[
<\left(\#F\right)^{2}\frac{\varepsilon}{8\left(\#F\right)^{2}}\#\widetilde{G}=\frac{\varepsilon}{8}\#\widetilde{G},
\]
whence 
\[
\#\left(\widetilde{G}\setminus G\right)<\left(\frac{1}{1-\frac{\varepsilon}{8}}-1\right)\#G<\frac{\varepsilon}{7}\#G.
\]

This estimate gives properties 3 and 4. 

For $g\in C$ one has 
\[
\#\left(gG\bigtriangleup G\right)=2\#\left(gG\setminus G\right)\leqslant2\#\left(g\widetilde{G}\setminus\widetilde{G}\right)+2\#\left\{ \widetilde{G}\setminus G\right\} <
\]
\[
<\frac{2\varepsilon}{8}\#\widetilde{G}+\frac{2\varepsilon}{7}\#G<\varepsilon\#G.
\]
 \end{proof}
\begin{lemma}
\label{lem:knmn}For arbitrary finite subsets $I,H\subset\mathcal{G}$
there exists an infinite collection $\left\{ h_{i}\right\} _{1}^{\infty}\subset\mathcal{G}$
such that for every $i$ 
\[
h_{i}^{-1}gh_{i}\notin H\setminus\left\{ g\right\} ,
\]
 if $g\in I$. \end{lemma}
\begin{proof}[Proof (by inductive construction).]
 Let $H_{1}\subset\mathcal{G}$ be an infinite subset. If there are
no $g\in I$ such that $h^{-1}gh\in H\setminus\left\{ g\right\} $
for infinitely many $h\in H_{1}$ then the collection 
\[
H_{1}\setminus\cup_{g\in I}\left\{ h\mid h^{-1}gh\in H\setminus\left\{ g\right\} \right\} 
\]
 is a required one. Otherwise there exist $g_{1}\in I$ with infinitely
many inclusions $h^{-1}g_{1}h\in H\setminus\left\{ g_{1}\right\} $,
$h\in H_{1}$, and an infinite subset $F_{1}\subset H_{1}$ such that
\[
f_{1,i}^{-1}g_{1}f_{1,i}=f_{1,j}^{-1}g_{1}f_{1,j}
\]
 for all $f_{1,i},f_{1,j}\in F_{1}$. Hence every $h\in H_{2}:=F_{1}F_{1}^{-1}$
commutes with $g_{1}$ and thus 
\[
h^{-1}g_{1}h\notin H\setminus\left\{ g_{1}\right\} .
\]

If for every $g\in I\setminus\left\{ g_{1}\right\} $ inclusion $h^{-1}gh\in H\setminus\left\{ g\right\} $
does not hold for infinitely many $h\in H_{2}$, then 
\[
H_{2}\setminus\cup_{g\in I\setminus\left\{ g_{1}\right\} }\left\{ h\mid h^{-1}gh\in H\setminus\left\{ g\right\} \right\} 
\]
is the required collection. In the case there exists $g_{2}\in I\setminus\left\{ g_{1}\right\} $
with infinitely many inclusions $h^{-1}g_{2}h\in H\setminus\left\{ g_{2}\right\} $, $h\in H_{2}$,
one can choose an infinite subset $F_{2}\subset H_{2}$ such that
\[
f_{2,i}^{-1}g_{2}f_{2,i}=f_{2,j}^{-1}g_{2}f_{2,j}
\]
 for all pairs $f_{2,i},f_{2,j}\in F_{2}$. Hence every $h\in H_{3}:=F_{2}F_{2}^{-1}$
commutes with $g_{1}$ and $g_{2}$.

Doing so we either obtain the required collection or some infinite
set $H_{n}$ in the centralizer of $\left\{ g_{1},\ldots,g_{n-1}\right\} $.
If the set 
\[
B_{n}=H_{n}\setminus\cup_{g\in I\setminus\left\{ g_{1},\ldots,g_{n-1}\right\} }\left\{ h\mid h^{-1}gh\in H\setminus\left\{ g\right\} \right\} 
\]
 is finite then $B_{n}\setminus H_{n}$ is a required collection.
Otherwise there exist $g_{n}\in I\setminus\left\{ g_{1},\ldots,g_{n-1}\right\} $
with infinitely many inclusions $h^{-1}g_{n}h\in H\setminus\left\{ g_{n}\right\} $,
$h\in H_{n}$ and an infinite subset $F_{n}\subset H_{n}$ such that
\[
f_{n,i}^{-1}g_{n}f_{n,i}=f_{n,j}^{-1}g_{n}f_{n,j}
\]
for all pairs $f_{n,i},f_{n,j}\in F_{n}$. Hence every $h\in H_{n+1}:=F_{n}F_{n}^{-1}$commutes
with $g_{1},\ldots,g_{n}$.

Either this inductive procedure stops (whence the lemma is already proved)
or infinite set $H_{\#I+1}$ (obtained as the result of the procedure)
is in the centralizer of $I$ whence $H_{\#I+1}$ is a required collection. \end{proof}
\begin{lemma}
\label{lem:gi}Let $\left\{ C_{i}\right\} _{i\in\mathbb{N}\cup\left\{ 0\right\} }$
be an increasing sequence of finite subsets in $\mathcal{G}$, satisfying
conditions: 
\[
C_{i}=C_{i}^{-1},
\]
for all $i$ and 
\[
C_{i}\setminus\left(C_{i-1}\right)^{5}\neq\textrm{Ø},
\]
 for $i\in\mathbb{N}$. 

Then for arbitrary collection $\left\{ g_{i}\mid g_{i}\in C_{i}\setminus\left(C_{i-1}\right)^{5}\right\} $
and fixed $g\in\mathcal{G}$, there are at most 2 solutions of the
inclusion 
\[
g_{i}gg_{j}^{-1}\in C_{0},i\neq j.
\]
\end{lemma}
\begin{proof}
Assume that 
\[
g_{i}gg_{j}^{-1},g_{l}gg_{k}^{-1}\in C_{0}.
\]
Let $l$ be the maximal of the indices $i,j,k,l$ (other variants
are considered similarly). 

One has inclusions 
\[
g_{i}\in C_{0}g_{j}g^{-1},g_{l}\in C_{0}g_{k}g^{-1},
\]
 and $g_{l}g_{i}^{-1}\in C_{0}g_{k}g^{-1}gg_{j}^{-1}C_{0}=C_{0}g_{k}g_{j}^{-1}C_{0}$
whence $g_{l}\in C_{0}g_{k}g_{j}^{-1}C_{0}g_{i}$. If $i,j,k\in\left[0,l-1\right]$$ $,
then $C_{0}g_{k}g_{j}^{-1}C_{0}g_{i}\subset\left(\text{С}_{l-1}\right)^{5}$
in contradiction with the choice of $g_{l}$. 

Hence either $l=i$ or $l=j$. If $l=i$, one has $g_{j}\in C_{0}^{-1}g_{i}g$,
and on the other hand $g_{k}\in C_{0}^{-1}g_{i}g$, it follows, that
$g_{j}g_{k}^{-1}\in C_{0}^{-1}C_{0}$, whence $k=j$ in view of the
choice of a sequence $\left\{ g_{i}\right\} $. The case $l=j$ is
considered similarly and in this case $k=i$. Thus besides $\left(g_{i},g_{j}\right)$
another solution of inclusion $g_{i}gg_{j}^{-1}\in C_{0}$ may be
$\left(g_{j},g_{i}\right)$. 
\end{proof}
Now for convenience reasons we introduce some notations. Inequality
$\left|a-b\right|<\varepsilon$ will be written as proximity $a\overset{\varepsilon}{\sim}b$. 

If an action $S$ is fixed, $A,B\in\Sigma$, $f,g,h\in\mathcal{G}$
we write 
\begin{equation}
k=f^{-1}gh,\tilde{A_{h}}=S^{h^{-1}}A,\tilde{B}_{f}=S^{f^{-1}}B.\label{eq:obozn}
\end{equation}

\begin{claim}
\label{utv1}For $\varepsilon>0$,mixing $\mathcal{G}$-action $S$,
finite $F\subset\mathcal{G}$ and finite $r\subset\Sigma$ there exists
an infinite subset $\left\{ g_{i}\right\} \subset\mathcal{G}$ such
that given $h,f\in F$, arbitrary $A,B\in r$ and $g\in\mathcal{G}$,
the proximity 
\begin{equation}
\mu\left(S^{g_{j}^{-1}kg_{i}}\tilde{A_{h}}\cap\tilde{B}_{f}\right)\overset{\varepsilon}{\sim}\mu\left(A\right)\mu\left(B\right),\label{eq:4}
\end{equation}
holds for all pairs $\left(i,j\right)$, $i\neq j$, except at most
two.

Moreover there is finite $C\subset\mathcal{G}$ such that given $h,f\in F$
arbitrary $A,B\in r$ and $g\in\mathcal{G}\setminus C$, the proximity
\begin{equation}
\mu\left(S^{g_{i}^{-1}kg_{i}}\tilde{A_{h}}\cap\tilde{B}_{f}\right)\overset{\varepsilon}{\sim}\mu\left(A\right)\mu\left(B\right),\label{eq:3}
\end{equation}
holds for all $i$, except at most one. \end{claim}
\begin{proof}
Let $N$ be such that 
\[
\left|g\right|>N-2\max_{f\in F}\left|f\right|\Rightarrow\mu\left(S^{g}A\cap B\right)\overset{\varepsilon}{\sim}\mu\left(A\right)\mu\left(B\right)
\]
 for any $A,B\in r$. Take $\left\{ g\in\mathcal{G}\mid\left|g\right|\leqslant N\right\} $
as $C$ and set $H=\bigcap_{f,h\in F}f^{-1}Ch$. One has that $g\notin H$
implies the proximities 
\begin{equation}
\forall f,h\in F,A,B\in r,:\mu\left(S^{g}\tilde{A_{h}}\cap\tilde{B}_{f}\right)\overset{\varepsilon}{\sim}\mu\left(\tilde{A_{h}}\right)\mu\left(\tilde{B}_{f}\right)=\mu\left(A\right)\mu\left(B\right),\label{eq:1234}
\end{equation}
 and moreover $C=C^{-1}$.

The required sequence $\left\{ g_{i}\right\} $ and auxiliary sequence
of subsets $C_{i}\subset\mathcal{G}$, $g_{i}\in C_{i}$, will
be constructed inductively. 

We set $C_{0}=H$ and let $g_{1}$ be an arbitrary element outside
$C_{0}^{5}$. As $C_{1}$ may be taken finite set, containing $g_{1}$,
$C_{0}^{5}$ and invariant with respect to inversion $g\rightarrow g^{-1}$.

Let $g_{1},\ldots,g_{n-1}$ and $C_{n-1}$ are already chosen. Applying
lemma \ref{lem:knmn} to the sets $H$ and $I_{n}=\bigcup_{i<n}g_{i}Hg_{i}^{-1}$
choose elements $\left\{ h_{i}\right\} $. Next as $g_{n}$ take a
arbitrary element of set $\left\{ h_{i}\right\} \setminus C_{n-1}^{5}$,
and then as $C_{n}$ take arbitrary finite set, containing $g_{n}$,
$C_{n-1}^{5}$ and invariant with respect to inversion.

This inductive procedure results in the increasing sequence of finite
sets $\left\{ C_{i}\mid C_{i}=C_{i}^{-1}\right\} $ and infinite sequence
of elements $\left\{ g_{i}\mid g_{i}\in C_{i}\setminus C_{i-1}^{5}\right\} $.

According to lemma \ref{lem:gi} given $k$ one has $g_{j}^{-1}kg_{i}\notin C_{0}=H$
for all (except at most 2) pairs $\left(i,j\right),i\neq j$. Hence
proximity (\ref{eq:4}) follows from (\ref{eq:1234}) and the first conclusion
of the claim is checked.

As regards the second conclusions of the claim we show that $\#\left\{ i\mid g_{i}^{-1}kg_{i}\in H\right\} \leqslant1$.
In fact, if $g_{i}^{-1}kg_{i}\in H$ then for every $n>i$ one has
$k\in I_{n}$ according to the choice of $I_{n}$. Since $g_{n}\in\left\{ h_{j}\right\} $
it follows (from lemma \ref{lem:knmn}) that $g_{n}^{-1}kg_{n}\notin H\setminus\left\{ k\right\} $.
Finally, $k\notin H\Leftarrow g\notin C$. 

Now application of proximity in (\ref{eq:1234}) finishes the proof.
\end{proof}

\section*{Construction of the special $\mathcal{G}$-action}

Fix two free mixing $\mathcal{G}$-actions $S:X\rightarrow X$ and
$Q:Y\rightarrow Y$, collection of sets $r\subset\Sigma$, Rokhlin
tower $\left\{ Q^{g}E\right\} _{g\in G}^{\varepsilon}$, finite subsets
$F,\left\{ g_{i}\right\} ,\left\{ c_{i}\right\} $ in $\mathcal{G}$,
such, that $\sqcup_{i}Fc_{i}=G$. Take an invertible measure-preserving
mapping $V:X\mapsto X\times Y$ so that $A\times Y$ for every $A\in r$.
Set $Z=V^{-1}J^{-1}\left(S\times Q\right)JV$ where the transformation
$J$ is given by the formulas 
\[
J\mid_{X\times O}\left(x,y\right)=\left(x,y\right)
\]
 and 
\[
J\mid_{X\times\left(Q^{fc_{i}}E\right)}\left(x,y\right)=\left(S^{fg_{i}f^{-1}}x,y\right)
\]
 for all $c_{i}$ and all $f\in F$ (recall that $O$ denotes the
remainder of the tower $\left\{ Q^{g}E\right\} _{g\in G}^{\varepsilon}$).

$Z$ is called a \emph{special action}.

Remark that for arbitrary $A,B\in r$, $g\in\mathcal{G}$ and an action
$T$ one has 
\[
\left|\mu\left(T^{g}A\cap B\right)-\mu\left(V^{-1}J^{-1}\left(S^{g}\times Q^{g}\right)JVA\cap B\right)\right|=
\]
\[
=\left|\mu\left(T^{g}A\cap B\right)-\mu\otimes\mu\left(J^{-1}\left(S^{g}\times Q^{g}\right)J\left(A\times Y\right)\cap\left(B\times Y\right)\right)\right|=
\]
\[
=\left|\mu\left(T^{g}A\cap B\right)-\right.
\]
\[
\left.-\mu\otimes\mu\left(J^{-1}\left(S^{g}\times Q^{g}\right)J\left(A\times\left(\sqcup_{i,h}Q^{hc_{i}}E\sqcup O\right)\right)\cap\left(B\times\left(\sqcup_{j,f}Q^{fc_{j}}E\sqcup O\right)\right)\right)\right|\leqslant
\]
\[
\leqslant\sum_{f,h,i,j}\left|\mu\left(T^{g}A\cap B\right)-\mu\left(S^{g_{j}^{-1}f^{-1}ghg_{i}h^{-1}}A\cap S^{f^{-1}}B\right)\right|\mu\left(Q^{ghc_{i}}E\cap Q^{fc_{j}}E\right)+2\mu\left(O\right).
\]

Using notations (\ref{eq:obozn}), this inequality can be written
in the form

\begin{equation}
\begin{array}{c}
\left|\mu\left(T^{g}A\cap B\right)-\mu\left(Z^{g}A\cap B\right)\right|<\\
<\sum_{f,h,i,j}\left|\mu\left(T^{g}A\cap B\right)-\mu\left(S^{g_{j}^{-1}kg_{i}}\tilde{A_{h}}\cap\tilde{B}_{f}\right)\right|\mu\left(Q^{ghc_{i}}E\cap Q^{fc_{j}}E\right)+2\varepsilon
\end{array}\label{eq:Glavn}
\end{equation}

The formula enables one to estimate the distance between $\mathcal{G}$-action
$T$ and special action $Z$. 
\begin{theorem}
\label{thm:Decart ploten}Let $\mathcal{G}$ be finitely generated
monotilable amenable group. For any free mixing $\mathcal{G}$-actions
$S$ and $Q$ one has $\overline{S\times Q}=\mathcal{M}_{\mathcal{G}}$.\end{theorem}
\begin{proof}
Given $\mathcal{G}$-action $T$ and positive number $\Delta$, we
show that in $\left(m,\Delta\right)$-neighbourhood of $T$ there
is an action conjugate to $S\times Q$.

First choose collection of measurable sets $r$ and $\varepsilon>0$
so that

\begin{equation}
R\in\mathcal{Q}\left(T,r,6\varepsilon\right)\Rightarrow m\left(R,T\right)<\Delta.\label{eq:thm0}
\end{equation}

Next fix a finite set $H\subset\mathcal{G}$ satisfying proximity
\[
\mu\left(T^{g}A\cap B\right)\overset{\varepsilon}{\sim}\mu\left(A\right)\mu\left(B\right)
\]
 for all $A,B\in r$, $g\notin H$ and choose a tile $F$ so that
$\#\left(gF\bigtriangleup F\right)<\varepsilon\#F$ for $g\in H$. 

Passing to conjugate action one can assume that 
\[
S^{g}\in\mathcal{U}\left(T^{g},r,\frac{\varepsilon}{2}\right),
\]
for every $g\in F$ (see remark after lemma \ref{lem:Alpern}). 

Using the previous claim choose corresponding elements $\left\{ g_{i}\right\} $
and set $C$. Next by $Q,F,C$ and $\varepsilon$ choose sets $G\subset\mathcal{G}$,$E\subset Y$
and sequence $\left\{ c_{i}\right\} $ as in lemma \ref{lemm:alpern2}.

Let $Z$ be the special action corresponding to the parameters obtained
above. To show that $Z\in\mathcal{Q}\left(T,r,6\varepsilon\right)$
we use inequality (\ref{eq:Glavn}).

Given $A,B\in r$ and $g\in\mathcal{G}$ let a quadruple $\left(i,j,f,h\right)$
be called 'bad' if 
\[
\left|\mu\left(T^{g}A\cap B\right)-\mu\left(S^{g_{j}^{-1}kg_{i}}\tilde{A_{h}}\cap\tilde{B}_{f}\right)\right|>2\varepsilon
\]
 (remember notations $k,\tilde{A_{h}},\tilde{B}_{f}$). The points
of $Q^{ghc_{i}}E\cap Q^{fc_{j}}E$ are called 'bad' if the corresponding
quadruple $\left(i,j,f,h\right)$ is 'bad'.

We now estimates the measure of 'bad' points in the cases: $g\in H;g\in C\setminus H;g\notin C$.

1. Let $g\in H$. If in the first case $h,gh\in F$ then the set $Q^{ghc_{i}}$
is of the form $Q^{fc_{j}}$ hence $i=j$, $f=gh$. Since $k=e$ one
has 
\[
\mu\left(S^{g_{j}^{-1}kg_{i}}\tilde{A_{h}}\cap\tilde{B}_{f}\right)=\mu\left(\tilde{A_{h}}\cap\tilde{B}_{f}\right)=
\]
\[
=\mu\left(S^{fh^{-1}}A\cap B\right)=\mu\left(S^{g}A\cap B\right)\overset{\varepsilon}{\sim}\mu\left(T^{g}A\cap B\right).
\]

Thus, only the points of those $Q^{ghc_{i}}E\cap Q^{fc_{j}}E$ may
be 'bad' for which $gh\notin F$. Measure of corresponding set does
not exceed the ratio $\frac{\#\left(gF\bigtriangleup F\right)}{\#F}$
which is less than $\varepsilon$ by the choice of $F$. 

2. Let $g\in C\setminus H$. If a quadruple $\left(i,j,f,h\right)$ is 
such that $h,gh\in F$ then all the points of $Q^{ghc_{i}}E\cap Q^{fc_{j}}E$
are not 'bad' (as above). Given a quadruple $\left(i,j,f,h\right)$
such that $h\in F$, $ghc_{i}\in G\setminus Fc_{i}$ one has $i\neq j$
and according to claim \ref{utv1} given $f$, $h$ and all pairs
$\left(i,j\right)$, except at most 2 
\[
\left|\mu\left(S^{g_{j}^{-1}kg_{i}}\tilde{A_{h}}\cap\tilde{B}_{f}\right)-\mu\left(T^{g}A\cap B\right)\right|\leqslant
\]
\[
\leqslant\left|\mu\left(S^{g_{j}^{-1}kg_{i}}\tilde{A_{h}}\cap\tilde{B}_{f}\right)-\mu\left(A\right)\mu\left(B\right)\right|+\left|\mu\left(A\right)\mu\left(B\right)-\mu\left(T^{g}A\cap B\right)\right|<2\varepsilon.
\]

Thus, only the points of those $Q^{ghc_{i}}E\cap Q^{fc_{j}}E$ may
be 'bad' for which either $ghc_{i}\notin G$ or the pairs $\left(i,j\right)$
are exceptional. In the first case measure of corresponding set does
not exceed the ratio $\frac{\#\left(gG\bigtriangleup G\right)}{\#G}$
which is less than $\varepsilon$ according to conclusion 2 of lemma
\ref{lemm:alpern2}. In the second case measure of the set of 'bad'
points does not exceed 
\[
\sum_{f,h}\sum_{\left(i,j\right)\in J}\mu\left(Q^{ghc_{i}}E\cap Q^{fc_{j}}E\right),
\]
 where $J=J(f,h)$ is the set of exceptional pairs $\left(i,j\right)$.
Since $\#J\leqslant2$ this measure can be estimated as follows 
\[
\sum_{f,h}2\mu\left(E\right)<\frac{2\left(\#F\right)^{2}}{\#G}<\varepsilon.
\]

3. Assume that $g\notin C$. According to claim for fixed $f,h$ and
all pairs $\left(i,j\right)$, except at most 3 one has 

\[
\left|\mu\left(S^{g_{j}^{-1}kg_{i}}\tilde{A_{h}}\cap\tilde{B}_{f}\right)-\mu\left(T^{g}A\cap B\right)\right|\leqslant
\]
\[
\leqslant\left|\mu\left(S^{g_{j}^{-1}kg_{i}}\tilde{A_{h}}\cap\tilde{B}_{f}\right)-\mu\left(A\right)\mu\left(B\right)\right|+\left|\mu\left(A\right)\mu\left(B\right)-\mu\left(T^{g}A\cap B\right)\right|<2\varepsilon,
\]
Hence the points of $Q^{ghc_{i}}E\cap Q^{fc_{j}}E$ are not 'bad'.

Measure of the set of 'bad' points in this case does not exceed  
\[
\sum_{f,h}\sum_{\left(i,j\right)\in J}\mu\left(Q^{ghc_{i}}E\cap Q^{fc_{j}}E\right),
\]
where $J=J(f,h)$ is the set of exceptional pairs. Since $\#J\leqslant3$,
for the total measure of 'bad' points can be estimated by 
\[
\sum_{f,h}3\mu\left(E\right)<\frac{3\left(\#F\right)^{2}}{\#G}<\varepsilon.
\]

So the right hand side in (\ref{eq:Glavn}) does not exceed the total
measure of 'bad' points plus $2\varepsilon$ one has that for all $g$
\[
\left|\mu\left(T^{g}A\cap B\right)-\mu\left(Z^{g}A\cap B\right)\right|
\]
 does not exceed $6\varepsilon$.\end{proof}
\begin{corollary}
Let $\mathcal{G}$ be finitely generated infinite monotilable group.
Then for every Bernoulli action of this group conjugates are everywhere dense
in $\mathcal{M}_{\mathcal{G}}$.
\end{corollary}
In fact theorem \ref{thm:Decart ploten} can be applied since Bernoulli
$\mathcal{G}$-action are Cartesian squares (see \cite{OrnsteinWeiss87}).
\begin{corollary}
Let $\mathcal{G}$ be finitely generated infinite monotilable group.
For every free mixing $\mathcal{G}$-action of positive entropy its
conjugates are everywhere dense in $\mathcal{M}_{\mathcal{G}}$.
\end{corollary}
It follows from lemma  \ref{lem:factorDence} and the fact that free
$\mathcal{G}$-action of positive entropy possesses a Bernoulli factor
\cite{OrnsteinWeiss87}.

\section{Concluding remark}
\label{sec:4}
1. Resides weak and lead metrics in the set of $\mathcal{G}$-actions
one can introduce mixed metrics. For example, metrics 
\[
d\left(T,S\right)+\sup_{n\in\mathbb{N}}a\left(T^{ng},S^{ng}\right)
\]
is complete and separable when considered in the set of $\mathcal{G}$-actions
$T$ with mixing $T^{g}$. Such a metric was used in \cite{TSVMSB2011}
for constructing of mixing transformations of homogeneous spectral
multiplicity $n>2$.

2. The lead metrics may be introduced for actions of locally compact
groups. In the case 

\[
m\left(T,S\right)=\sum_{i}\frac{1}{2^{i}}\sup_{g\in K_{i}}d\left(T^{g},S^{g}\right)+\sup_{g\in\mathcal{G}}a\left(T^{g},S^{g}\right),
\]
 where $\left\{ K_{i}\right\} $ is a increasing to $\mathcal{G}$
sequence of compact sets.

\end{document}